\pdfoutput=1
\documentclass[11pt,reqno,a4paper]{amsart}

\usepackage{iftex}
\ifPDFTeX
  \usepackage[utf8]{inputenc}
  \usepackage[T1]{fontenc}
\fi
\usepackage[british]{babel}
\usepackage[a4paper,margin=1.12in]{geometry}
\usepackage{microtype}
\usepackage{mathpazo}
\usepackage{euler}
\usepackage{amsmath,amssymb,amsthm,mathtools,mathrsfs}
\usepackage{aliascnt}
\usepackage[dvipsnames]{xcolor}
\usepackage{enumitem}
\usepackage{hyperref}
\hypersetup{
  colorlinks=true,
  linkcolor=blue!60!black,
  citecolor=blue!60!black,
  urlcolor=blue!60!black
}
\usepackage[capitalize,nameinlink]{cleveref}

\numberwithin{equation}{section}
\theoremstyle{plain}
\newtheorem{theorem}{Theorem}[section]
\newaliascnt{lemma}{theorem}
\newtheorem{lemma}[lemma]{Lemma}
\aliascntresetthe{lemma}
\newaliascnt{proposition}{theorem}
\newtheorem{proposition}[proposition]{Proposition}
\aliascntresetthe{proposition}
\newaliascnt{corollary}{theorem}
\newtheorem{corollary}[corollary]{Corollary}
\aliascntresetthe{corollary}

\theoremstyle{definition}
\newaliascnt{definition}{theorem}
\newtheorem{definition}[definition]{Definition}
\aliascntresetthe{definition}
\newaliascnt{example}{theorem}
\newtheorem{example}[example]{Example}
\aliascntresetthe{example}
\newaliascnt{assumption}{theorem}

\aliascntresetthe{assumption}

\theoremstyle{remark}
\newaliascnt{remark}{theorem}

\aliascntresetthe{remark}

\Crefname{theorem}{Theorem}{Theorems}
\Crefname{lemma}{Lemma}{Lemmas}
\Crefname{proposition}{Proposition}{Propositions}
\Crefname{corollary}{Corollary}{Corollaries}
\Crefname{definition}{Definition}{Definitions}
\Crefname{example}{Example}{Examples}
\Crefname{assumption}{Assumption}{Assumptions}
\Crefname{remark}{Remark}{Remarks}

\newcommand{\C}{\mathbb C}
\newcommand{\cO}{\mathcal O}
\newcommand{\cI}{\mathcal I}
\newcommand{\cR}{\mathcal R}
\newcommand{\cA}{\mathcal A}
\newcommand{\OmegaLog}{\Omega_X^1(\log D)}
\newcommand{\TLog}{T_X(-\log D)}
\newcommand{\Fix}{\operatorname{Fix}}
\newcommand{\Res}{\operatorname{Res}}
\newcommand{\Tr}{\operatorname{Tr}}
\newcommand{\Spec}{\operatorname{Spec}}
\newcommand{\ord}{\operatorname{ord}}
\newcommand{\Ind}{\operatorname{Ind}}
\newcommand{\Lef}{\operatorname{Lef}}
\newcommand{\DR}{\operatorname{DR}}
\newcommand{\Sp}{\operatorname{Sp}}
\newcommand{\SpTr}{\operatorname{SpTr}}
\newcommand{\LT}{\operatorname{LT}}
\newcommand{\length}{\operatorname{length}}
\newcommand{\dd}{\mathrm d}

\title[Logarithmic Lefschetz formulae]
{Logarithmic Lefschetz fixed point formulae and resonant boundary indices}

\author{Elaheh Shahsavaripour}
\address{Dipartimento di Matematica, Universit\`a degli Studi di Bari Aldo Moro, Via E. Orabona 4, I-70125 Bari, Italy}
\email{shahsavaripour.mat.it@gmail.com}



\begin{document}

\begin{abstract}
We study boundary contributions to the holomorphic Lefschetz formula for strict self-maps of compact complex manifolds with a simple normal crossings divisor. At a normally non-resonant boundary fixed point, the local term reduces to the fixed-point problem on the boundary stratum and vanishes in the de Rham specialisation. For resonant points satisfying a normal admissibility condition, finite flat rescaling and relative duality recover the local trace as the coefficient at the excess normal order of a holomorphic family of relative residue traces; in the balanced case, this coefficient is the trace of a class on the special fibre. The numerical coefficient is independent of the admissible presentation. In a coupled family, the de Rham boundary index is the product of the normal contact order and the tangential complete-intersection multiplicity. Boundary contributions to the fixed-point formula for the complement are supported on the resonant fixed points.
\end{abstract}

\maketitle

\section{Introduction}

The holomorphic Lefschetz fixed point formula expresses a global trace on coherent cohomology as a sum of local duality invariants of the intersection of the graph of a holomorphic map with the diagonal. Let $X$ be a compact complex manifold, let $f:X\to X$ be holomorphic, and let $\Phi:f^*E\to E$ be a linearisation of a holomorphic vector bundle. If the fixed-point set is finite, then the local term at $p\in\Fix(f)$ is the Grothendieck residue
\begin{equation}\label{eq:intro-classical-residue}
\LT_p(f,E,\Phi)
=
\Res_p
\left[
\frac{\tau_{\Phi,p}\,\dd x_1\wedge\cdots\wedge\dd x_n}
{x_1-f_1,\ldots,x_n-f_n}
\right],
\end{equation}
where $\tau_{\Phi,p}$ is the intrinsic trace class of the linearisation in the local fixed-point algebra. The simple fixed-point formula goes back to Atiyah--Bott; residue formulations covering isolated degenerate fixed points were developed subsequently by Sibner--Sibner, Toledo--Tong and O'Brian; see \cite{AtiyahBottII,SibnerSibner,ToledoTong1,ToledoTong2,OBrian}. A direct geometric proof of Lefschetz fixed point formulae via equivariant index theory was given by Lafferty, Yu and Zhang \cite{LaffertyYuZhang}. A categorical proof of the classical holomorphic Atiyah--Bott formula was given by Kondyrev and Prikhodko \cite{KondyrevPrikhodko}.

Residual index theory for holomorphic self-maps provides a complementary antecedent. Abate introduced a residual index for germs tangent to the identity \cite{AbateResidual}; Bracci and Tovena extended the index theory to maps fixing a possibly singular curve on a surface \cite{BracciTovena}, and Abate, Bracci and Tovena subsequently established index theorems for self-maps with positive-dimensional fixed-point sets \cite{AbateBracciTovena}. In the latter setting the distinguished normal geometry comes from the fixed locus. Here the fixed-point scheme is zero-dimensional; the distinguished normal directions come instead from the boundary divisor, and resonance forces higher-order information to survive in the local Lefschetz residue.

This paper studies the local terms when $X$ is equipped with a simple normal crossings divisor $D$ and $f$ is a strict self-map of the pair. Near a point on $k$ branches of $D$, choose adapted coordinates $(z_1,\ldots,z_k,w_1,\ldots,w_m)$ with $D=\{z_1\cdots z_k=0\}$. Strictness gives
\[
f^*z_i=z_i a_i(z,w),\qquad a_i\in\cO_{X,p}^{\times},
\]
so the fixed equations in the normal directions are $z_i(1-a_i)$. The divisor therefore enters the local problem both through these factors and through the action of $f$ on $\Omega_X^1(\log D)$.

There is a parallel infinitesimal picture. If $V\in T_X(-\log D)$ has local flow $\varphi_t$, then $V(z_i)\in(z_i)$ implies $\varphi_t^*z_i=z_i a_i(t,z,w)$ with $a_i$ invertible, so the local flow is strict. At a zero $p$, one has $d\varphi_{t,p}=\exp(t\,dV_p)$ and $x-\varphi_t(x)=-tV(x)+O(t^2)$; the normal multipliers are therefore obtained by exponentiating the infinitesimal normal linearisation. Bott and O'Brian developed residue formulae for holomorphic vector fields \cite{BottVectorFields,OBrian}. Aleksandrov introduced the logarithmic index through the contraction complex of $\Omega_X^\bullet(\log D)$ \cite{AleksandrovLogIndex}. Corr\^ea and Machado obtained global residue and Poincar\'e--Hopf type formulae for one-dimensional foliations logarithmic along normal crossings divisors \cite{CorreaMachado2019,CorreaMachado2024}, while Corr\^ea, Louren\c{c}o and Machado introduced logarithmic Baum--Bott residues for free divisors, generalising Aleksandrov's logarithmic index \cite{CorreaLourencoMachado}. These results give the infinitesimal counterpart to the self-map residue problem; the present paper studies in addition the resonance of the normal multipliers.

For an $f$-linearised bundle $(E,\Phi)$, consider the polynomial obtained from the exterior powers of the logarithmic cotangent bundle. Writing $J_f^{\log}$ for the induced endomorphism of $\Omega_X^1(\log D)$ over the fixed-point scheme, the classical residue formula applied degree by degree gives
\begin{equation}\label{eq:intro-y-formula}
L_y^{\log}(f,E,\Phi;D)
=
\sum_{p\in\Fix(f)}
\Res_p
\left[
\frac{\tau_{\Phi,p}\det(1+yJ_f^{\log})\,
\dd x_1\wedge\cdots\wedge\dd x_n}
{x_1-f_1,\ldots,x_n-f_n}
\right].
\end{equation}
Formula \eqref{eq:intro-y-formula} is \Cref{thm:y-lefschetz}. It is a direct consequence of the holomorphic Lefschetz theorem; the substantive point is the analysis of its boundary terms under the normal filtration defined by $D$.

The non-resonant case reduces to the fixed-point problem on the boundary stratum. If $p$ lies on $k$ branches of $D$ and the normal linearisation has no eigenvalue $1$, then the local polynomial in \eqref{eq:intro-y-formula} is divisible by $(1+y)^k$ and is an iterated residue on the closed stratum through $p$; see \Cref{thm:nonresonant-reduction}. Hence every normally non-resonant boundary point contributes zero at $y=-1$. In the same normally non-resonant case, the specialisation at $y=0$ gives the holomorphic local Lefschetz term as a residue on the stratum. The complementary case, in which $1$ occurs in the normal spectrum, is called normally resonant.

The specialisation $y=-1$ computes the Lefschetz number on the complement $U=X\setminus D$; this is \Cref{thm:de-rham-lefschetz}, and the resulting interior--boundary decomposition is \Cref{cor:open-boundary-formula}. Thus the boundary part of the fixed-point formula for $f|_U$ is supported on resonant fixed points, while the normally non-resonant boundary terms disappear. In one normal dimension, if $f(z)=za(z)$ and the origin is an isolated fixed point, \Cref{prop:curve-contact} gives
\begin{equation*}
\Ind^{\mathrm{dR},\log}_{0,D}(f)=\ord_0(1-a).
\end{equation*}
It therefore records the order of contact of the normal multiplier with the identity.

At a resonant point, the preceding full reduction to the boundary stratum need not be available. The branch ideals instead define an intrinsic multifiltration of the local fixed-point algebra. Under a normal admissibility condition on a regular-sequence presentation of the fixed-point ideal, normal rescaling produces a finite flat complete-intersection family; see \Cref{lem:finite-flat-family}. The relative trace is constructed in \Cref{thm:relative-trace}. The normal-specialisation formula, \Cref{thm:normal-cone-specialisation}, shows that the coefficient at the excess normal order recovers the original local residue. When the transformed numerator is divisible by the required power $t^{e(q)}$, \Cref{cor:normal-cone-class} identifies this coefficient with the trace of a class on the special fibre. In general, the normal-cone class is attached to the chosen admissible one-parameter family, whereas the numerical coefficient is independent of the presentation and of the adapted coordinates by \Cref{prop:numerical-canonicity}.

The preceding constructions may be summarised as follows.

\begin{theorem}\label{thm:main-intro}
Let $(X,D)$ be a compact smooth pair with simple normal crossings boundary, let
$f:(X,D)\to(X,D)$ be a strict logarithmic self-map with finite fixed-point set, and let
$(E,\Phi)$ be an $f$-linearised holomorphic vector bundle.
\begin{enumerate}[label=\textup{(\roman*)}]
\item At a normally non-resonant boundary fixed point lying on $k$ branches of $D$, the local logarithmic Lefschetz polynomial reduces to the fixed-point residue on the closed stratum and is divisible by $(1+y)^k$. In particular, its de Rham specialisation vanishes.
\item At a normally admissible boundary fixed point, the local fixed-point trace is recovered as the coefficient at the excess normal order of the relative residue trace associated with the finite flat normal rescaling. This coefficient is independent of the admissible presentation and of the adapted coordinates. If the transformed numerator is divisible by the corresponding power $t^{e(q)}$, the coefficient is the trace of a residue class on the special fibre.
\item Put $U=X\setminus D$. For trivial coefficients, the specialisation at $y=-1$ is the Lefschetz number of $f|_U$. Consequently, only normally resonant boundary fixed points can contribute to the boundary part of the fixed-point formula for the complement.
\end{enumerate}
\end{theorem}

Part~\textup{(i)} is \Cref{thm:nonresonant-reduction,cor:boundary-vanishing}; Part~\textup{(ii)} follows from \Cref{lem:finite-flat-family,thm:relative-trace,thm:normal-cone-specialisation,prop:numerical-canonicity}; and Part~\textup{(iii)} is \Cref{thm:de-rham-lefschetz,cor:open-boundary-formula}.

The construction also applies when the normal and tangential equations are coupled. \Cref{prop:higher-dimensional-contact} gives a family with local fixed-point algebra of length $(r+1)m_1\cdots m_m$ and logarithmic de Rham index $rm_1\cdots m_m$. Thus the resonant index is the product of the normal contact order and the multiplicity of the tangential complete-intersection special fibre, even though the original fixed equations do not split into normal and tangential factors. The compact examples in \Cref{sec:examples} exhibit simultaneously interior, resonant boundary and normally non-resonant boundary contributions.

Logarithmic Lefschetz classes also occur in arithmetic ramification theory. Kato and Saito obtain trace formulae for open varieties from logarithmic intersections with the diagonal \cite{KatoSaitoPerfect,KatoSaitoLocal}. Their setting is $\ell$-adic and intersection-theoretic, whereas here the setting is holomorphic and the local invariant is the Grothendieck residue functional of an isolated fixed-point algebra, together with its behaviour under the normal filtration. A different categorical direction is provided by logarithmic Hochschild cohomology and homology for log-smooth pairs \cite{HablicsekHerrLeonardi} and by the functorial logarithmic correspondences of \cite{GyengeHablicsekHerr}. The residue construction below is also complementary to logarithmic Bott localisation for vector fields with non-isolated local complete-intersection zero varieties \cite{CorreaShahsavaripour}.

The normal-specialisation theorem is proved for normally admissible isolated fixed points. A presentation-free multifiltered extension beyond this class would naturally be formulated using the Rees complex of the complete fixed-point ideal rather than the initial forms of a chosen regular sequence.

\Cref{sec:local-duality} recalls the local residue form of the holomorphic Lefschetz theorem. The boundary multifiltration is introduced in \Cref{sec:multifiltration}, and strict logarithmic self-maps and the logarithmic cotangent action are treated in \Cref{sec:log-data}. The polynomial fixed-point formula and the non-resonant reduction occupy \Cref{sec:y-formula,sec:nonresonant}. Normal rescaling and relative duality are developed in \Cref{sec:normal-cone}; the holomorphic and resonant specialisations follow in \Cref{sec:holomorphic-specialisation,sec:resonance}. Finally, \Cref{sec:de-rham} identifies the de Rham specialisation with the Lefschetz number of the complement, and \Cref{sec:examples} gives compact examples.

\section{Local fixed-point duality}\label{sec:local-duality}

Let \(M\) be a complex manifold of dimension \(n\), let \(p\in M\), and let
$
g_1,\ldots,g_n\in\cO_{M,p}
$
have an isolated common zero at \(p\). The quotient
\begin{equation*}
A_g=\cO_{M,p}/(g_1,\ldots,g_n)
\end{equation*}
is an Artinian complete-intersection algebra. If \(x_1,\ldots,x_n\) are local coordinates, the Grothendieck residue is the functional
\begin{equation*}
\lambda_g:A_g\longrightarrow\C,
\qquad
\overline h\longmapsto
\Res_p
\left[
\frac{h\,\dd x_1\wedge\cdots\wedge\dd x_n}{g_1,\ldots,g_n}
\right].
\end{equation*}
It is the local duality trace of the zero-dimensional complete intersection. The following form of the transformation law will be used.

\begin{lemma}\label{lem:transformation-law}
Suppose that \(q_i=\sum_j a_{ij}g_j\), where \(q_1,\ldots,q_n\) also have an isolated common zero at \(p\). Then
\[
\Res_p
\left[
\frac{h\,\dd x_1\wedge\cdots\wedge\dd x_n}{g_1,\ldots,g_n}
\right]
=
\Res_p
\left[
\frac{h\det(a_{ij})\,\dd x_1\wedge\cdots\wedge\dd x_n}{q_1,\ldots,q_n}
\right].
\]
\end{lemma}

\begin{proof}
This is the standard transformation law for Grothendieck residues. It follows either from the multidimensional Cauchy integral or from the functoriality of the local duality trace; see \cite[Chapter~5]{GriffithsHarris} and \cite{HartshorneResidues}.
\end{proof}

\begin{corollary}\label{cor:units}
If \(u_1,\ldots,u_n\in\cO_{M,p}^{\times}\), then
\[
\Res_p
\left[
\frac{h\,\dd x_1\wedge\cdots\wedge\dd x_n}{u_1g_1,\ldots,u_ng_n}
\right]
=
\Res_p
\left[
\frac{h(u_1\cdots u_n)^{-1}\,\dd x_1\wedge\cdots\wedge\dd x_n}{g_1,\ldots,g_n}
\right].
\]
\end{corollary}

The following lemma gives the reduction along a coordinate complete intersection.

\begin{lemma}\label{lem:iterated-reduction}
Let \((z_1,\ldots,z_k,w_1,\ldots,w_m)\) be coordinates at \(p\), and put \(S=\{z_1=\cdots=z_k=0\}\). Let \(q_1,\ldots,q_m\) be holomorphic germs such that
\[
h_\alpha(w)=q_\alpha(0,w)
\]
have an isolated common zero on \(S\). Then
\begin{align*}
&\Res_p
\left[
\frac{H(z,w)\,\dd z_1\wedge\cdots\wedge\dd z_k\wedge
\dd w_1\wedge\cdots\wedge\dd w_m}
{z_1,\ldots,z_k,q_1,\ldots,q_m}
\right]
\\
&\hspace{35mm}=
\Res_p^S
\left[
\frac{H(0,w)\,\dd w_1\wedge\cdots\wedge\dd w_m}
{h_1,\ldots,h_m}
\right].
\end{align*}
\end{lemma}

\begin{proof}
Write
\[
q_\alpha(z,w)=h_\alpha(w)+\sum_{i=1}^kz_i c_{\alpha i}(z,w).
\]
The systems \((z_1,\ldots,z_k,q_1,\ldots,q_m)\) and \((z_1,\ldots,z_k,h_1,\ldots,h_m)\) are related by a block triangular matrix of determinant one. By \Cref{lem:transformation-law}, the first residue is unchanged when the \(q_\alpha\) are replaced by the \(h_\alpha\). Successive application of the one-variable Cauchy formula in the variables \(z_1,\ldots,z_k\) proves the result.
\end{proof}

Let \(f:X\to X\) be holomorphic, and let \(p\) be an isolated fixed point. In local coordinates, put
\begin{equation*}
I_{f,p}=(x_1-f_1,\ldots,x_n-f_n),
\qquad
A_{f,p}=\cO_{X,p}/I_{f,p}.
\end{equation*}
Write \(\lambda_{f,p}:A_{f,p}\to\C\) for the Grothendieck residue functional associated with the fixed-point sequence \((x_1-f_1,\ldots,x_n-f_n)\).

\begin{lemma}\label{lem:coordinate-invariance-residue}
The functional \(\lambda_{f,p}\) is independent of the choice of local coordinates.
\end{lemma}

\begin{proof}
Let \(y=\varphi(x)\) be another coordinate system and write
\[
 g=x-f(x),\qquad q=\varphi(x)-\varphi(f(x)).
\]
Since each component of \(q\) vanishes on the diagonal \(x=f(x)\), there is a holomorphic matrix \(A(x)\) such that \(q=A(x)g\). Modulo \(I_{f,p}\), the matrix \(A\) is congruent to \(D\varphi\): for instance, take
\[
 A_{ij}(x)=\int_0^1
 \frac{\partial\varphi_i}{\partial x_j}
 \bigl(f(x)+s(x-f(x))\bigr)\,\dd s.
\]
Hence \(\det A\equiv\det D\varphi\pmod{I_{f,p}}\), and both determinants are units. For a germ \(h\), the transformation law gives
\[
 \Res_p\left[\frac{h\,\dd y_1\wedge\cdots\wedge\dd y_n}{q_1,\ldots,q_n}\right]
 =
 \Res_p\left[
 \frac{h\,(\det D\varphi/\det A)\,\dd x_1\wedge\cdots\wedge\dd x_n}{g_1,\ldots,g_n}
 \right].
\]
The quotient \(\det D\varphi/\det A\) represents \(1\) in \(A_{f,p}\), so the last residue equals \(\lambda_{f,p}([h])\).
\end{proof}

Thus \(\lambda_{f,p}\) is the intrinsic local duality trace of the fixed-point algebra. Let \(\Phi:f^*E\to E\) be a linearisation. A local frame of \(E\) represents \(\Phi\) by a matrix \(M(x)\). Its trace determines a class
\begin{equation}\label{eq:trace-class}
\tau_{\Phi,p}=[\Tr M]\in A_{f,p}.
\end{equation}

\begin{lemma}\label{lem:trace-class}
The class \(\tau_{\Phi,p}\) is independent of the local frame of \(E\).
\end{lemma}

\begin{proof}
Under a change of frame by a holomorphic matrix \(G\), the matrix of the linearisation changes by
\[
M'(x)=G(x)^{-1}M(x)G(f(x)).
\]
In the quotient \(A_{f,p}\), \(G(f(x))=G(x)\). Thus \(M'\) is conjugate to \(M\) over \(A_{f,p}\), and their traces have the same class.
\end{proof}

Set
\begin{equation*}
L(f,E,\Phi)
=
\sum_{q\geq0}(-1)^q
\Tr\bigl(H^q(\Phi)\circ f^*\bigr).
\end{equation*}
The residue form of the holomorphic Lefschetz theorem may therefore be written intrinsically as follows.

\begin{theorem}\label{thm:classical-lefschetz}
Let \(X\) be a compact complex manifold, let \(f:X\to X\) be holomorphic with finite fixed-point set, and let \(\Phi:f^*E\to E\) be a linearisation. Then
\begin{equation*}
L(f,E,\Phi)
=
\sum_{p\in\Fix(f)}
\Res_p
\left[
\frac{\tau_{\Phi,p}\,\dd x_1\wedge\cdots\wedge\dd x_n}
{x_1-f_1,\ldots,x_n-f_n}
\right].
\end{equation*}
\end{theorem}

\begin{proof}
This is the holomorphic Lefschetz theorem in its local duality form; for isolated degenerate fixed points see in particular \cite{SibnerSibner,ToledoTong2,OBrian}. Coordinate independence is \Cref{lem:coordinate-invariance-residue}, while \Cref{lem:trace-class} gives independence of the frame of \(E\).
\end{proof}

\section{Boundary multifiltrations}\label{sec:multifiltration}

Let \(p\in D\) lie on the local branches \(D_1,\ldots,D_k\), and let
\[
R=\cO_{X,p},
\qquad
A_{f,p}=R/I_{f,p}.
\]
Write \(I_i\subset R\) for the ideal of \(D_i\) at \(p\). The boundary determines a decreasing multifiltration of the local fixed-point algebra.

\begin{definition}\label{def:boundary-multifiltration}
For \(\boldsymbol\nu=(\nu_1,\ldots,\nu_k)\in\mathbb N^k\), set
\begin{equation}\label{eq:boundary-multifiltration}
F_D^{\boldsymbol\nu}A_{f,p}
=
\operatorname{im}
\left(
I_1^{\nu_1}\cdots I_k^{\nu_k}
\longrightarrow A_{f,p}
\right).
\end{equation}
The corresponding multi-Rees algebra is
\begin{equation*}
\cR_D(A_{f,p})
=
\bigoplus_{\boldsymbol\nu\in\mathbb N^k}
F_D^{\boldsymbol\nu}A_{f,p}\,\mathbf T^{\boldsymbol\nu},
\qquad
\mathbf T^{\boldsymbol\nu}=T_1^{\nu_1}\cdots T_k^{\nu_k}.
\end{equation*}
\end{definition}

The triple
\begin{equation*}
\bigl(A_{f,p},\lambda_{f,p},F_D^\bullet\bigr)
\end{equation*}
is the local duality datum together with the filtration imposed by the boundary.

\begin{proposition}\label{prop:intrinsic-multifiltration}
The multifiltration \eqref{eq:boundary-multifiltration} is independent of adapted coordinates. A permutation of the local branches permutes its indices. It is also compatible with isomorphisms of germs of strict logarithmic pairs carrying the fixed-point scheme and the boundary branches to one another.
\end{proposition}

\begin{proof}
Each term in \eqref{eq:boundary-multifiltration} is defined from the intrinsic branch ideals and the quotient map \(R\to A_{f,p}\). An isomorphism of germs carries these ideals and the fixed-point ideal to their counterparts. This proves the proposition.
\end{proof}

A positive integral weight \(\mathbf c=(c_1,\ldots,c_k)\) converts the multifiltration into the decreasing filtration
\begin{equation*}
F_{D,\mathbf c}^{\ell}A_{f,p}
=
\sum_{\mathbf c\cdot\boldsymbol\nu\geq\ell}
F_D^{\boldsymbol\nu}A_{f,p}.
\end{equation*}
In adapted coordinates, this filtration is realised by the rescaling
\[
(z_1,\ldots,z_k,w)
\longmapsto
(t^{c_1}z_1,\ldots,t^{c_k}z_k,w).
\]
Only the equal-weight case \(c_1=\cdots=c_k=1\) is needed below. For this filtration, \(\operatorname{in}_D(h)\) denotes the lowest non-zero normal homogeneous part of a germ \(h\), and \(\operatorname{in}_D(I)\) is the ideal generated by \(\operatorname{in}_D(h)\) for \(h\in I\). The coordinate formula for \(\operatorname{in}_D(h)\) is recalled in \Cref{sec:normal-cone}.

The initial forms of a displayed set of fixed-point equations need not generate the complete initial ideal. The following germ will be used again in \Cref{sec:resonance}.

\begin{example}\label{ex:hidden-initial-relation}
Let
$f(z,w)=\bigl(z(1-w),z^2\bigr),
\qquad
D=\{z=0\}.
$
The map is strict, and its fixed-point ideal is
\begin{equation*}
I_{f,0}=(zw,w-z^2).
\end{equation*}
The origin is isolated and \(A_{f,0}\simeq\C\{z\}/(z^3)\). With respect to the \(z\)-adic order, the initial forms of the displayed generators generate only \((w)\). Nevertheless,
\[
zw-z(w-z^2)=z^3,
\]
so the complete initial ideal is
\begin{equation*}
\operatorname{in}_D(I_{f,0})=(w,z^3).
\end{equation*}
Equivalently, the same fixed-point ideal has the regular-sequence presentation
\begin{equation}\label{eq:hidden-admissible-presentation}
q_1=z^3,
\qquad
q_2=w-z^2,
\end{equation}
whose initial forms have finite colength. Thus relations among the fixed-point equations in the chosen coordinates may have to be incorporated before normal specialisation.
\end{example}

\section{Strict logarithmic self-maps}\label{sec:log-data}

Let \(X\) be a complex manifold and let
$
D=D_1+\cdots+D_r
$
be a simple normal crossings divisor. The logarithmic cotangent bundle is locally free. In coordinates
$
D=\{z_1\cdots z_k=0\},
$
it has frame
\begin{equation}\label{eq:log-frame}
\frac{\dd z_1}{z_1},\ldots,\frac{\dd z_k}{z_k},
\dd w_1,\ldots,\dd w_m.
\end{equation}
Its dual is \(\TLog\). The notation is standard; see \cite{DeligneRegular,Saito}.

\begin{definition}\label{def:strict-map}
A holomorphic map \(f:X\to X\) is a strict logarithmic self-map of \((X,D)\) if, for every irreducible component \(D_i\) of \(D\), the ideal sheaves satisfy
\[
f^{-1}\cI_{D_i}\cdot\cO_X=\cI_{D_i}.
\]
Equivalently, \(f^*D_i=D_i\) as effective Cartier divisors, and every adapted chart satisfies
\begin{equation*}
f^*z_i=z_i a_i(z,w),
\qquad
a_i\in\cO_X^{\times}.
\end{equation*}
\end{definition}

A strict self-map preserves \(U=X\setminus D\) and induces a bundle morphism
\begin{equation}\label{eq:log-differential}
\dd f_{\log}^*:f^*\OmegaLog\longrightarrow\OmegaLog.
\end{equation}
Indeed,
\begin{equation}\label{eq:dlog-pullback}
f^*\left(\frac{\dd z_i}{z_i}\right)
=
\frac{\dd z_i}{z_i}+\frac{\dd a_i}{a_i}.
\end{equation}
Let \(Z_f\) denote the fixed-point scheme. Restriction to \(Z_f\) identifies the source and target of \eqref{eq:log-differential}, and gives an endomorphism
\begin{equation*}
J_f^{\log}:\OmegaLog|_{Z_f}\longrightarrow\OmegaLog|_{Z_f}.
\end{equation*}
Its characteristic polynomial has coefficients in \(\cO_{Z_f}\).
Let \(p\in D\) lie on precisely \(k\) local components and let
\[
S_p=\{z_1=\cdots=z_k=0\}
\]
be the corresponding closed stratum germ. In adapted coordinates, write
\begin{equation}\label{eq:local-map-components}
f^*z_i=z_i a_i(z,w),
\qquad
f^*w_\alpha=F_\alpha(z,w).
\end{equation}
The restriction \(f_{S_p}\) is given by \(w\mapsto F(0,w)\). The differential of \(f\) induces a normal bundle morphism
\[
\dd_N f:N_{S_p/X}\longrightarrow f_{S_p}^*N_{S_p/X},
\]
which is diagonal in the local frame determined by the \(z_i\), with multipliers \(a_i(0,w)\).

\begin{lemma}\label{lem:block-log-jacobian}
After restriction to the fixed-point scheme of \(f_{S_p}\), the logarithmic cotangent endomorphism is block triangular with respect to the decomposition determined by the frame \eqref{eq:log-frame}. Up to transposition, according to the row or column convention for matrices of bundle morphisms, it has the form
\[
J_f^{\log}|_{S_p}
=
\begin{pmatrix}
I_k & B\\
0 & J_{f_{S_p}}
\end{pmatrix}
\]
for a holomorphic matrix \(B\), where \(J_{f_{S_p}}\) is the cotangent differential of the restricted map. Consequently, in the local fixed-point algebra of \(f_{S_p}\),
\begin{equation}\label{eq:block-determinant}
\det(1+yJ_f^{\log})|_{S_p}
=
(1+y)^k\det(1+yJ_{f_{S_p}}).
\end{equation}
\end{lemma}

\begin{proof}
Equation \eqref{eq:dlog-pullback} shows that the coefficient of \(\dd z_j/z_j\) in \(\dd a_i/a_i\) is divisible by \(z_j\) and hence vanishes on \(S_p\). Thus the normal logarithmic block is the identity. Moreover,
\[
\dd F_\alpha
=
\sum_j\frac{\partial F_\alpha}{\partial z_j}z_j\frac{\dd z_j}{z_j}
+
\sum_\beta\frac{\partial F_\alpha}{\partial w_\beta}\dd w_\beta,
\]
so the normal logarithmic part of \(\dd F_\alpha\) also vanishes on \(S_p\). The remaining tangential block is the cotangent differential of \(f_{S_p}\). The determinant identity follows.
\end{proof}

\section{The logarithmic \texorpdfstring{$y$}{y}-Lefschetz formula}\label{sec:y-formula}

Put \(n=\dim X\). Let \((E,\Phi)\) be an \(f\)-linearised holomorphic vector bundle. For every \(r\), the tensor product
\[
E\otimes\Omega_X^r(\log D)
\]
has the induced linearisation
\begin{equation}\label{eq:induced-linearisation}
f^*E\otimes f^*\Omega_X^r(\log D)
\xrightarrow{\ \Phi\otimes\wedge^r\dd f_{\log}^*\ }
E\otimes\Omega_X^r(\log D).
\end{equation}

\begin{definition}\label{def:y-lefschetz}
Let \(\Phi_r\) denote the morphism in \eqref{eq:induced-linearisation}. The logarithmic \(y\)-Lefschetz polynomial is
\begin{equation*}
L_y^{\log}(f,E,\Phi;D)
=
\sum_{r=0}^{n}y^r
\sum_{q\geq0}(-1)^q
\Tr\bigl(H^q(\Phi_r)\circ f^*\bigr).
\end{equation*}
For trivial coefficients, write
\[
L_y^{\log}(f;D)=L_y^{\log}(f,\cO_X,\operatorname{id};D).
\]
\end{definition}

At an isolated fixed point \(p\), let \(\tau_{\Phi,p}\) be the class \eqref{eq:trace-class}. The trace of the tensor-product linearisation on the fixed-point algebra is
\[
\tau_{\Phi,p}\Tr(\wedge^rJ_f^{\log}).
\]
The elementary identity
\begin{equation}\label{eq:exterior-determinant}
\sum_{r=0}^{n}y^r\Tr(\wedge^rA)=\det(1+yA)
\end{equation}
gives the polynomial formula.

\begin{theorem}\label{thm:y-lefschetz}
Let \((X,D)\) be a compact smooth pair with simple normal crossings boundary, let \(f:(X,D)\to(X,D)\) be a strict logarithmic self-map with finite fixed-point set, and let \((E,\Phi)\) be an \(f\)-linearised holomorphic vector bundle. Then
\begin{equation}\label{eq:y-lefschetz-residue}
L_y^{\log}(f,E,\Phi;D)
=
\sum_{p\in\Fix(f)}
\Res_p
\left[
\frac{
\tau_{\Phi,p}\det(1+yJ_f^{\log})\,
\dd x_1\wedge\cdots\wedge\dd x_n
}{x_1-f_1,\ldots,x_n-f_n}
\right].
\end{equation}
\end{theorem}

\begin{proof}
Apply \Cref{thm:classical-lefschetz} to the bundle \(E\otimes\Omega_X^r(\log D)\) for each \(r\), multiply by \(y^r\), and sum. The local numerator is the trace of the tensor product in \eqref{eq:induced-linearisation}. Formula \eqref{eq:exterior-determinant} gives \eqref{eq:y-lefschetz-residue}.
\end{proof}

Over the interior \(U=X\setminus D\), the logarithmic cotangent endomorphism is the ordinary cotangent endomorphism. Denote its restriction to the local fixed-point algebra by
\[
J_f:\Omega_X^1|_{Z_f\cap U}\longrightarrow\Omega_X^1|_{Z_f\cap U}.
\]

\begin{definition}\label{def:interior-boundary-indices}
For \(p\in\Fix(f)\setminus D\), define
\begin{equation*}
\Ind_p^{\mathrm{int}}(f,E,\Phi;y)
=
\lambda_{f,p}
\left(
\tau_{\Phi,p}\det(1+yJ_f)
\right).
\end{equation*}
For \(p\in\Fix(f)\cap D\), define
\begin{equation*}
\Ind_{p,D}^{\mathrm{bdry}}(f,E,\Phi;y)
=
\lambda_{f,p}
\left(
\tau_{\Phi,p}\det(1+yJ_f^{\log})
\right).
\end{equation*}
For trivial coefficients, we omit \(E\) and \(\Phi\) and write
\[
\Ind_p^{\mathrm{int}}(f;y)
=\Ind_p^{\mathrm{int}}(f,\cO_X,\operatorname{id};y),
\qquad
\Ind_{p,D}^{\mathrm{bdry}}(f;y)
=\Ind_{p,D}^{\mathrm{bdry}}(f,\cO_X,\operatorname{id};y).
\]
\end{definition}

These numbers are intrinsic: both the trace class and the cotangent endomorphism are defined on the local fixed-point algebra.

\begin{corollary}\label{cor:interior-boundary-splitting}
Under the hypotheses of \Cref{thm:y-lefschetz},
\begin{equation*}
L_y^{\log}(f,E,\Phi;D)
=
\sum_{p\in\Fix(f)\setminus D}
\Ind_p^{\mathrm{int}}(f,E,\Phi;y)
+
\sum_{p\in\Fix(f)\cap D}
\Ind_{p,D}^{\mathrm{bdry}}(f,E,\Phi;y).
\end{equation*}
\end{corollary}

At \(y=0\), \Cref{thm:y-lefschetz} gives the ordinary holomorphic Lefschetz number:
\begin{equation*}
L_0^{\log}(f,E,\Phi;D)=L(f,E,\Phi).
\end{equation*}
The corresponding local term is \eqref{eq:intro-classical-residue}.

\section{Normally non-resonant boundary points}\label{sec:nonresonant}

Let \(p\in D\) lie on \(k\) branches. Use the notation \eqref{eq:local-map-components} and put
\begin{equation*}
b_i=1-a_i,
\qquad
g_i=z_ib_i,
\qquad
q_\alpha=w_\alpha-F_\alpha(z,w).
\end{equation*}
The fixed equations on the stratum are
\begin{equation*}
h_\alpha(w)=w_\alpha-F_\alpha(0,w).
\end{equation*}
We write
\[
\Omega_{S_p}=\dd w_1\wedge\cdots\wedge\dd w_m,
\]
and \(\Res_p^{S_p}\) denotes the Grothendieck residue on the stratum \(S_p\), computed in the coordinates \(w_1,\ldots,w_m\).
When the normal factors are units, the local fixed-point algebra is canonically identified with
\[
\cO_{S_p,p}/(h_1,\ldots,h_m),
\]
and \(\tau_{\Phi|S_p}\) denotes the image of \(\tau_{\Phi,p}\) under this identification.
The normal linearisation has multipliers \(a_i(0,w)\). Its Lefschetz determinant along the fixed-point scheme of \(f_{S_p}\) is therefore
\begin{equation}\label{eq:normal-determinant}
\det(1-\dd_Nf)=\prod_{i=1}^k b_i(0,w).
\end{equation}

\begin{definition}\label{def:normal-nonresonance}
The fixed point \(p\) is normally non-resonant if
\[
\det(1-\dd_Nf_p)\neq0.
\]
\end{definition}

Under this condition, the germs \(b_i\) are units. Since \(I_{f,p}\) is primary for the maximal ideal \(\mathfrak m_p\subset\cO_{X,p}\), the equations \(h_1,\ldots,h_m\) have an isolated common zero on \(S_p\).

\begin{theorem}\label{thm:nonresonant-reduction}
Let \(p\in D\) be a normally non-resonant isolated fixed point. Then its local polynomial in \Cref{thm:y-lefschetz} is
\begin{align}\label{eq:nonresonant-y-reduction}
\Ind_{p,D}^{\mathrm{bdry}}(f,E,\Phi;y)
&=
(1+y)^k
\Res_p^{S_p}
\left[
\frac{
\tau_{\Phi|S_p}
\det(1+yJ_{f_{S_p}})
\det(1-\dd_Nf)^{-1}\,\Omega_{S_p}
}{h_1,\ldots,h_m}
\right].
\end{align}
\end{theorem}

\begin{proof}
The local term in \eqref{eq:y-lefschetz-residue} is
\[
\Res_p
\left[
\frac{
\tau_{\Phi,p}\det(1+yJ_f^{\log})\,
\dd z_1\wedge\cdots\wedge\dd z_k\wedge\Omega_{S_p}
}{z_1b_1,\ldots,z_kb_k,q_1,\ldots,q_m}
\right].
\]
By \Cref{cor:units}, the factors \(b_i\) may be transferred to the numerator as \((b_1\cdots b_k)^{-1}\). Apply \Cref{lem:iterated-reduction}. Restriction to \(S_p\), together with \eqref{eq:block-determinant} and \eqref{eq:normal-determinant}, gives \eqref{eq:nonresonant-y-reduction}.
\end{proof}

\begin{corollary}\label{cor:y-zero-nonresonant}
Under the hypotheses of \Cref{thm:nonresonant-reduction}, the specialisation at \(y=0\) is
\begin{equation*}
\LT_p(f,E,\Phi)
=
\Res_p^{S_p}
\left[
\frac{
\tau_{\Phi|S_p}\det(1-\dd_Nf)^{-1}\,\Omega_{S_p}
}{h_1,\ldots,h_m}
\right].
\end{equation*}
\end{corollary}

\begin{corollary}\label{cor:boundary-vanishing}
Let \(p\in D\) be an isolated fixed point. If \(p\) is normally non-resonant, then
\[
\Ind_{p,D}^{\mathrm{bdry}}(f,E,\Phi;-1)=0.
\]
More precisely, the local polynomial is divisible by \((1+y)^k\).
\end{corollary}

\begin{corollary}\label{cor:nondegenerate-formula}
Let \(p\in D\) be an isolated fixed point. If \(p\) is normally non-resonant and non-degenerate for \(f_{S_p}\), then
\[
\Ind_{p,D}^{\mathrm{bdry}}(f,E,\Phi;y)
=
(1+y)^k
\frac{
\Tr(\Phi_p)\det(1+y\dd f_{S_p,p}^*)
}{
\det(1-\dd_Nf_p)\det(1-\dd f_{S_p,p})
}.
\]
\end{corollary}

\section{Normal-cone specialisation}\label{sec:normal-cone}

Before performing the normal rescaling, we allow an arbitrary regular-sequence presentation of the complete fixed-point ideal. Let
\[
R=\C\{z_1,\ldots,z_k,w_1,\ldots,w_m\},
\qquad n=k+m,
\]
and write
\[
\dd z=\dd z_1\wedge\cdots\wedge\dd z_k,
\qquad
\dd w=\dd w_1\wedge\cdots\wedge\dd w_m.
\]
Let \(J=I_{S_p,p}=(z_1,\ldots,z_k)\) be the ideal of the closed stratum through \(p\). The equal-weight normal filtration is the \(J\)-adic filtration.

\begin{lemma}\label{lem:intrinsic-normal-order}
For every non-zero germ \(h\in R\),
\[
\nu_D(h)=\max\{\ell:h\in J^\ell\}.
\]
Its normal initial form is the image of \(h\) in \(\operatorname{gr}_J^{\nu_D(h)}R\). Thus the normal order and the associated graded initial ideal are invariant under changes of adapted coordinates.
\end{lemma}

\begin{proof}
In adapted coordinates, \(J^\ell\) consists exactly of the germs whose monomials in the normal variables have total degree at least \(\ell\), which gives the displayed equality. An adapted coordinate change preserves the stratum ideal \(J\), hence induces an isomorphism of the \(J\)-adic filtered local rings and of their associated graded rings. It therefore preserves normal orders and the corresponding initial ideals.
\end{proof}

For a non-zero germ \(h\in R\), write
\[
h(z,w)=\sum_{\alpha\in\mathbb N^k}h_\alpha(w)z^\alpha
\]
and define
\begin{equation*}
\nu_D(h)=\min\{|\alpha|:h_\alpha\neq0\}.
\end{equation*}
Equivalently, the normal initial form introduced in \Cref{sec:multifiltration} is
\begin{equation*}
\operatorname{in}_D(h)
=
\sum_{|\alpha|=\nu_D(h)}h_\alpha(w)z^\alpha.
\end{equation*}

Let \(g_1,\ldots,g_n\) be the coordinate fixed-point equations, put \(I=(g_1,\ldots,g_n)\), and let \(\mathfrak m\) be the maximal ideal of \(R\). The ideal \(I\) is an \(\mathfrak m\)-primary complete intersection. Hence every regular sequence \(q_1,\ldots,q_n\) generating \(I\) is minimal, and there is a matrix \(A=(a_{ij})\in\operatorname{GL}_n(R)\) such that
\begin{equation}\label{eq:change-fixed-generators}
(q_1,\ldots,q_n)^t=A(g_1,\ldots,g_n)^t.
\end{equation}
Set \(\Delta_q=\det A\in R^\times\).

\begin{definition}\label{def:admissible-presentation}
A regular-sequence presentation \(q=(q_1,\ldots,q_n)\) of \(I\) is normally admissible if
\begin{equation}\label{eq:admissible-colength}
R/\bigl(\operatorname{in}_D(q_1),\ldots,
\operatorname{in}_D(q_n)\bigr)
\end{equation}
has finite length and
\[
\sum_{j=1}^{n}\nu_D(q_j)\geq k.
\]
A boundary fixed point is normally admissible if its fixed-point ideal admits such a presentation in adapted coordinates.
\end{definition}

By \Cref{lem:intrinsic-normal-order}, this property is independent of the chosen adapted coordinates: an adapted change of coordinates preserves the normal orders and identifies the associated graded initial ideals, hence preserves finite colength and regular-sequence presentations.

Finite colength in \eqref{eq:admissible-colength} implies that the \(n\) initial forms constitute a regular sequence. The condition is therefore directly verifiable and is not an additional Cohen--Macaulay hypothesis.
Fix a normally admissible presentation. Put \(d_j=\nu_D(q_j)\) and define
\begin{equation}\label{eq:homogenised-equations}
Q_j(t,z,w)=t^{-d_j}q_j(tz,w).
\end{equation}
Then \(Q_j(0,z,w)=\operatorname{in}_D(q_j)\). Let \(\Delta\) be a small disc, let \(V\) be a sufficiently small polydisc, and let
\begin{equation*}
\cA_q
=
\cO_{\Delta\times V}/(Q_1,\ldots,Q_n).
\end{equation*}
Denote the corresponding analytic space by \(\mathcal Z_q\) and the projection by
\[
\pi_q:\mathcal Z_q\longrightarrow\Delta.
\]

\begin{lemma}\label{lem:finite-flat-family}
After shrinking \(\Delta\) and \(V\), the morphism \(\pi_q\) is finite and flat. The set-theoretic support of \(\mathcal Z_q\) is the zero section \(\Delta\times\{0\}\). The functions \(Q_1,\ldots,Q_n\) form a relative regular sequence, and every fibre is a zero-dimensional complete intersection of the same length.
\end{lemma}

\begin{proof}
Choose a polydisc \(V\) such that the original fixed-point ideal has no zero in \(\overline V\setminus\{0\}\), and shrink the parameter disc so that \(|t|<1\). Normal admissibility allows us, after shrinking \(V\) once more, to assume that the initial forms have no common zero in \(\overline V\setminus\{0\}\). Uniform convergence of \(Q_j(t,\cdot)\) to \(Q_j(0,\cdot)\) on \(\partial V\) then excludes common zeros on the boundary for sufficiently small \(t\). For \(t\neq0\), the equations \(Q_j(t,z,w)=0\) are equivalent to \(q_j(tz,w)=0\). Since \((tz,w)\in V\) and the original ideal has only the zero at the origin, this forces \((z,w)=(0,0)\). The same conclusion follows on the special fibre from finite colength of the initial ideal. Thus the set-theoretic support is \(\Delta\times\{0\}\). After shrinking \(\Delta\) and \(V\), the projection of the zero scheme to \(\Delta\) is therefore proper and has finite fibres; a proper holomorphic map with finite fibres is finite.

At every point of \(\mathcal Z_q\), the ambient local ring is regular of dimension \(n+1\), whereas the quotient by the \(n\) functions \(Q_j\) has dimension one. Their ideal therefore has height \(n\), so the \(Q_j\) form a regular sequence. The quotient is Cohen--Macaulay of dimension one. At a point lying over \(t_0\in\Delta\), the element \(t-t_0\) is a parameter because the fibre is zero-dimensional; hence it is a non-zero-divisor. The corresponding finite module over the discrete valuation ring \(\cO_{\Delta,t_0}\) is therefore torsion-free, and hence flat. Thus \(\cA_q\) is finite flat, so it is locally free over the base and all fibres have the same length.
\end{proof}

Write
\begin{equation*}
\omega_{\mathcal Z_q/\Delta}
=
\mathcal{E}xt^n_{\cO_{\Delta\times V}}
\bigl(\cO_{\mathcal Z_q},
\Omega^n_{\Delta\times V/\Delta}\bigr).
\end{equation*}
It is an invertible \(\cO_{\mathcal Z_q}\)-module, locally generated by the residue symbol
\[
\left[
\frac{\dd z\wedge\dd w}{Q_1,\ldots,Q_n}
\right].
\]

\begin{theorem}\label{thm:relative-trace}
There is a canonical morphism
\begin{equation*}
\Tr_{\pi_q}:
(\pi_q)_*\omega_{\mathcal Z_q/\Delta}
\longrightarrow\cO_\Delta
\end{equation*}
with the following properties.

\begin{enumerate}[label=\textup{(\roman*)},leftmargin=2.2em]
\item Finite flat duality gives
\[
(\pi_q)_*\omega_{\mathcal Z_q/\Delta}
\simeq
\mathcal{H}om_{\cO_\Delta}
\bigl((\pi_q)_*\cO_{\mathcal Z_q},\cO_\Delta\bigr),
\]
and \(\Tr_{\pi_q}\) corresponds to evaluation at the unit element \(1\in(\pi_q)_*\cO_{\mathcal Z_q}\).

\item For every holomorphic germ \(H(t,z,w)\),
\begin{equation}\label{eq:relative-residue-trace}
\Tr_{\pi_q}
\left(
H
\left[
\frac{\dd z\wedge\dd w}{Q_1,\ldots,Q_n}
\right]
\right)(t)
=
\Res_0
\left[
\frac{H(t,z,w)\,\dd z\wedge\dd w}
{Q_1(t),\ldots,Q_n(t)}
\right].
\end{equation}

\item The formation of both the dualising module and the trace is compatible with holomorphic base change; in particular, its restriction to each fibre is the Grothendieck duality trace of that fibre.
\end{enumerate}
\end{theorem}

\begin{proof}
Set \(B=(\pi_q)_*\cO_{\mathcal Z_q}\). By \Cref{lem:finite-flat-family}, \(B\) is locally free over \(\cO_\Delta\). Finite duality therefore identifies
\[
(\pi_q)_*\omega_{\mathcal Z_q/\Delta}
\simeq
\mathcal{H}om_{\cO_\Delta}(B,\cO_\Delta),
\]
and the trace is the counit \(\varphi\mapsto\varphi(1)\). For the relative complete intersection cut out by \(Q_1,\ldots,Q_n\), the Koszul description of the dualising module identifies the residue symbol with the relative Grothendieck residue. This gives \eqref{eq:relative-residue-trace}; see \cite{HartshorneResidues,LipmanResidues}.

If \(\Delta'\to\Delta\) is a holomorphic base change, then \(B'=B\otimes_{\cO_\Delta}\cO_{\Delta'}\) remains locally free and
\[
\mathcal{H}om_{\cO_\Delta}(B,\cO_\Delta)
\otimes_{\cO_\Delta}\cO_{\Delta'}
\simeq
\mathcal{H}om_{\cO_{\Delta'}}(B',\cO_{\Delta'}).
\]
Under this identification, evaluation at \(1\) and the Koszul residue symbol are preserved. Hence the dualising module and trace are compatible with holomorphic base change.
\end{proof}

For \(t_0\in\Delta\), write \(\mathcal Z_{q,t_0}=\pi_q^{-1}(t_0)\), let \(\omega_{\mathcal Z_{q,t_0}}\) be its dualising module obtained by base change, and denote the fibre trace furnished by \Cref{thm:relative-trace}\textup{(iii)} by
\[
\Tr_{\pi_q,t_0}:H^0(\mathcal Z_{q,t_0},\omega_{\mathcal Z_{q,t_0}})\longrightarrow\C.
\]

For later use, set
\begin{equation*}
\mathscr R_{q,H}(t)
=
\Res_0
\left[
\frac{H(t,z,w)\,\dd z\wedge\dd w}
{Q_1(t),\ldots,Q_n(t)}
\right].
\end{equation*}
This is holomorphic by \Cref{thm:relative-trace}. Put
\begin{equation}\label{eq:excess-order}
e(q)=\sum_{j=1}^{n}d_j-k.
\end{equation}

\begin{theorem}\label{thm:normal-cone-specialisation}
Let \(q\) be a normally admissible presentation, and let \(h\in R\). For \(t\neq0\),
\begin{equation}\label{eq:scaling-residue}
\mathscr R_{q,\,h(tz,w)\Delta_q(tz,w)}(t)
=
t^{e(q)}
\Res_0
\left[
\frac{h\,\dd z\wedge\dd w}{g_1,\ldots,g_n}
\right].
\end{equation}
Consequently,
\begin{equation}\label{eq:coefficient-specialisation}
\Res_0
\left[
\frac{h\,\dd z\wedge\dd w}{g_1,\ldots,g_n}
\right]
=
\frac{1}{e(q)!}
\left.
\frac{\dd^{e(q)}}{\dd t^{e(q)}}
\mathscr R_{q,\,h(tz,w)\Delta_q(tz,w)}(t)
\right|_{t=0}.
\end{equation}
If \(e(q)=0\), this becomes
\begin{equation}\label{eq:special-fibre-residue}
\Res_0
\left[
\frac{h\,\dd z\wedge\dd w}{g_1,\ldots,g_n}
\right]
=
\Res_0
\left[
\frac{h(0,w)\Delta_q(0,w)\,\dd z\wedge\dd w}
{\operatorname{in}_D(q_1),\ldots,
\operatorname{in}_D(q_n)}
\right].
\end{equation}
\end{theorem}

\begin{proof}
The transformation law applied to \eqref{eq:change-fixed-generators} gives
\[
\Res_0
\left[
\frac{h\,\dd z\wedge\dd w}{g_1,\ldots,g_n}
\right]
=
\Res_0
\left[
\frac{h\Delta_q\,\dd z\wedge\dd w}{q_1,\ldots,q_n}
\right].
\]
For \(t\neq0\), the defining relation \eqref{eq:homogenised-equations} gives
\[
\frac{1}{Q_1(t,z,w)\cdots Q_n(t,z,w)}
=
\frac{t^{\sum_jd_j}}
{q_1(tz,w)\cdots q_n(tz,w)}.
\]
After the change of variables \(z'=tz\), the relation \(\dd z=t^{-k}\dd z'\) therefore gives the factor \(t^{\sum_jd_j-k}=t^{e(q)}\). This proves \eqref{eq:scaling-residue}. The residue function on the left is holomorphic at the origin by \Cref{thm:relative-trace}. Since \(e(q)\geq0\) by normal admissibility, the identity extends across \(t=0\), and extraction of the coefficient of \(t^{e(q)}\) gives \eqref{eq:coefficient-specialisation}. The final assertion follows by base change at \(t=0\).
\end{proof}

\begin{definition}\label{def:specialised-trace-coefficient}
Denote the right-hand side of \eqref{eq:coefficient-specialisation} by
\begin{equation*}
\SpTr_{D,q}(h).
\end{equation*}
If
\begin{equation}\label{eq:balanced-numerator}
h(tz,w)\Delta_q(tz,w)
=t^{e(q)}\widetilde h(t,z,w),
\end{equation}
we define the associated normal-cone residue class by
\begin{equation*}
\Sp^{\log}_{D,q}(h)
=
\widetilde h(0,z,w)
\left[
\frac{\dd z\wedge\dd w}
{\operatorname{in}_D(q_1),\ldots,
\operatorname{in}_D(q_n)}
\right]
\in H^0\bigl(\mathcal Z_{q,0},\omega_{\mathcal Z_{q,0}}\bigr).
\end{equation*}
\end{definition}

\begin{corollary}\label{cor:normal-cone-class}
Under \eqref{eq:balanced-numerator},
\begin{equation*}
\Tr_{\pi_q,0}\bigl(\Sp^{\log}_{D,q}(h)\bigr)
=
\SpTr_{D,q}(h)
=
\Res_0
\left[
\frac{h\,\dd z\wedge\dd w}{g_1,\ldots,g_n}
\right].
\end{equation*}
\end{corollary}

\begin{proof}
Under \eqref{eq:balanced-numerator}, division of \eqref{eq:scaling-residue} by \(t^{e(q)}\) identifies the quotient with the relative residue of \(\widetilde h\). Base change at \(t=0\) in \Cref{thm:relative-trace} gives the result.
\end{proof}

\begin{proposition}\label{prop:numerical-canonicity}
Let \(q\) and \(q'\) be normally admissible presentations of the same fixed-point ideal. Then
\[
\SpTr_{D,q}(h)=\SpTr_{D,q'}(h)=\lambda_{f,p}([h]).
\]
Thus, when the class \([h]\in A_{f,p}\) is intrinsically defined, the specialisation coefficient is independent of the admissible presentation, the adapted coordinates and the chosen representative of that class. The normal-cone class itself is attached to the chosen admissible one-parameter family; here, only its trace is claimed to be canonical.
\end{proposition}

\begin{proof}
Both coefficients equal the Grothendieck residue \(\lambda_{f,p}([h])\) by \Cref{thm:normal-cone-specialisation}. Independence of the representative follows because the residue annihilates the fixed-point ideal, and coordinate independence is \Cref{lem:coordinate-invariance-residue}.
\end{proof}

The following criteria show that normal admissibility is not a restatement of the conclusion.

\begin{proposition}\label{prop:admissibility-criteria}
Let \(p\in D\) be an isolated fixed point.
\begin{enumerate}[label=\textup{(\roman*)},leftmargin=2.2em]
\item If \(p\) is normally non-resonant, the coordinate fixed-point equations form a normally admissible presentation, and their excess order is zero.
\item Suppose that, for integers \(r_1,\ldots,r_k\geq0\), after an adapted change of coordinates and multiplication by units, the fixed-point ideal has a presentation
\[
(z_1^{r_1+1},\ldots,z_k^{r_k+1},h_1(w),\ldots,h_m(w)),
\]
where \((h_1,\ldots,h_m)\) has finite colength. Then the presentation is normally admissible.
\item The presentation \eqref{eq:hidden-admissible-presentation} in \Cref{ex:hidden-initial-relation} is normally admissible, although the fixed-point equations in the chosen coordinates are not.
\end{enumerate}
\end{proposition}

\begin{proof}
In the first case, the fixed-point equations in the normal variables are unit multiples of \(z_i\), whereas the normal initial forms of the tangential equations are the fixed-point equations of the restricted map on the stratum. Their common zero is isolated, and the sum of the normal orders is \(k\); hence the coordinate presentation is normally admissible and has excess order zero.

For \textup{(ii)}, the displayed generators have normal initial forms
\[
z_1^{r_1+1},\ldots,z_k^{r_k+1},h_1(w),\ldots,h_m(w).
\]
Their quotient has finite length because \((h_1,\ldots,h_m)\) has finite colength, and the sum of the normal orders is \(k+r_1+\cdots+r_k\geq k\). Thus the presentation is normally admissible. Finally, \textup{(iii)} follows from \((\operatorname{in}_D q_1,\operatorname{in}_D q_2)=(z^3,w)\).
\end{proof}

\section{The holomorphic specialisation}\label{sec:holomorphic-specialisation}

At \(y=0\), the logarithmic polynomial recovers the ordinary holomorphic Lefschetz number. For a boundary fixed point, set
\begin{equation*}
\LT_{p,D}^{\log}(f,E,\Phi)
=
\Ind_{p,D}^{\mathrm{bdry}}(f,E,\Phi;0).
\end{equation*}
Numerically, this is the classical local Lefschetz term; the boundary filtration, however, provides a refined description of it.

\begin{theorem}\label{thm:holomorphic-specialisation}
Let \(p\in D\) be normally admissible, and let \(q\) be a normally admissible presentation. Then
\begin{equation}\label{eq:holomorphic-specialisation}
\LT_{p,D}^{\log}(f,E,\Phi)
=
\SpTr_{D,q}(\tau_{\Phi,p}).
\end{equation}
If \(e(q)=0\), choose a holomorphic representative \(\tau_\Phi(z,w)\) of \(\tau_{\Phi,p}\). The right-hand side is the trace of the special-fibre class
\[
\tau_\Phi(0,w)\Delta_q(0,w)
\left[
\frac{\dd z\wedge\dd w}
{\operatorname{in}_D(q_1),\ldots,
\operatorname{in}_D(q_n)}
\right].
\]
In the normally non-resonant case, this class gives the stratum residue in \Cref{cor:y-zero-nonresonant}.
\end{theorem}

\begin{proof}
Choose a holomorphic representative \(\tau_\Phi\) of \(\tau_{\Phi,p}\). By \Cref{thm:normal-cone-specialisation}, the specialised trace coefficient \(\SpTr_{D,q}(\tau_\Phi)\) equals the original Grothendieck residue, hence the local Lefschetz term at \(y=0\). Independence of the representative follows from \Cref{prop:numerical-canonicity}. If \(e(q)=0\), formula \eqref{eq:special-fibre-residue} gives the displayed special-fibre class. In the normally non-resonant case, \Cref{prop:admissibility-criteria}\textup{(i)} gives \(e(q)=0\), and \Cref{thm:nonresonant-reduction} identifies its trace with the stratum residue.
\end{proof}

\begin{corollary}\label{cor:global-holomorphic-splitting}
Under the hypotheses of \Cref{thm:y-lefschetz}, the holomorphic Lefschetz number decomposes as
\begin{equation*}
L(f,E,\Phi)
=
\sum_{p\in\Fix(f)\setminus D}\LT_p(f,E,\Phi)
+
\sum_{p\in\Fix(f)\cap D}\LT_{p,D}^{\log}(f,E,\Phi).
\end{equation*}
If every boundary fixed point is normally admissible, each term in the second sum is given by \eqref{eq:holomorphic-specialisation}.
\end{corollary}

\section{Resonant boundary indices}\label{sec:resonance}

For a boundary fixed point and trivial coefficients, write
\[
\Ind_{p,D}^{\mathrm{dR},\log}(f)
=
\Ind_{p,D}^{\mathrm{bdry}}(f,\cO_X,\operatorname{id};-1).
\]

At a normally admissible boundary fixed point, the local terms in \Cref{thm:y-lefschetz} may be specialised by \Cref{thm:normal-cone-specialisation}. A boundary fixed point is normally resonant if \(1\) is an eigenvalue of \(\dd_Nf_p\). The local residue remains well defined because the fixed-point scheme is zero-dimensional at the point, although the reduction in \Cref{thm:nonresonant-reduction} no longer applies.

The one-dimensional case is explicit.

\begin{proposition}\label{prop:curve-contact}
Let \(X\) be a complex curve, let \(D=\{p\}\), and let \(f:(X,D)\to(X,D)\) be strict. Choose a coordinate \(z\) centred at \(p\) and write
\[
f(z)=z a(z),
\qquad a(0)\neq0.
\]
If \(p\) is an isolated fixed point, then
\begin{equation}\label{eq:curve-contact-index}
\Ind^{\mathrm{dR},\log}_{p,D}(f)
=
\ord_p(1-a).
\end{equation}
\end{proposition}

\begin{proof}
The logarithmic cotangent line is generated by \(\dd z/z\), and the induced multiplier is
\[
J_f^{\log}(z)
=
\frac{zf'(z)}{f(z)}
=
1+\frac{za'(z)}{a(z)}.
\]
At \(y=-1\), the local term is
\begin{align*}
\Res_0
\left[
\frac{1-J_f^{\log}(z)}{z-f(z)}\,\dd z
\right]
&=
\Res_0
\left[
-\frac{a'(z)}{a(z)(1-a(z))}\,\dd z
\right].
\end{align*}
Put \(b=1-a\). If \(b\) is a unit, the displayed form is holomorphic and both sides of \eqref{eq:curve-contact-index} are zero. Otherwise \(a(0)=1\), and
\[
-\frac{a'}{a(1-a)}\,\dd z
=
\frac{b'}{b(1-b)}\,\dd z
=
\frac{b'}{b}\,\dd z
+
\frac{b'}{1-b}\,\dd z.
\]
The second summand is holomorphic, whereas the first has residue \(\ord_0(b)\). This proves the formula.
\end{proof}

For a single normal branch written as \(f^*z=za\), the integer \(\nu_D(1-a)\), when finite, will be called the normal contact order with the identity. In the one-dimensional situation of \Cref{prop:curve-contact}, it is \(\ord_p(1-a)\).

The same calculation factorises when the normal and tangential variables separate.

\begin{proposition}\label{prop:split-contact}
Let
\[
D=\{z_1\cdots z_k=0\}\subset(\C^{k+m},0),
\]
and suppose that
\[
f(z,w)
=
\bigl(z_1a_1(z_1),\ldots,z_ka_k(z_k),F(w)\bigr),
\]
where the origin is an isolated fixed point. Put
\[
r_i=\ord_0(1-a_i).
\]
Then
\begin{equation*}
\Ind^{\mathrm{dR},\log}_{0,D}(f)
=
\left(\prod_{i=1}^kr_i\right)
\Res_0
\left[
\frac{\det(1-J_F)\,\dd w_1\wedge\cdots\wedge\dd w_m}
{w_1-F_1,\ldots,w_m-F_m}
\right],
\end{equation*}
where \(J_F\) denotes the cotangent endomorphism induced by \(F\) on its local fixed-point algebra. If the fixed point of \(F\) is non-degenerate, the second factor is one.
\end{proposition}

\begin{proof}
The logarithmic differential is block diagonal, and the defining equations separate into normal and tangential variables. The product property of Grothendieck residues therefore factors the local term into the one-dimensional normal residues and the ordinary de Rham fixed-point residue of \(F\). Apply \Cref{prop:curve-contact} to each normal coordinate. Moreover, \(\det(1-J_F)\) is the Jacobian determinant of the tangential fixed-point sequence \(w-F(w)\); by the standard Jacobian formula for Grothendieck residues, its residue is the length of the local fixed-point algebra of \(F\); see \cite[Chapter~5]{GriffithsHarris}. In particular, it is one when that fixed point is non-degenerate.
\end{proof}

The example from \Cref{ex:hidden-initial-relation} gives a calculation in which the initial forms of the coordinate presentation are insufficient, whereas an alternative presentation is admissible.

\begin{proposition}\label{prop:hidden-relation-index}
For
$
f(z,w)=\bigl(z(1-w),z^2\bigr),
\ D=\{z=0\},
$
the origin is an isolated normally resonant fixed point, and
\begin{equation}\label{eq:hidden-local-polynomial}
\Ind_{0,D}^{\mathrm{bdry}}(f;y)=2y^2.
\end{equation}
Consequently,
\[
\LT_{0,D}^{\log}(f)=0,
\qquad
\Ind_{0,D}^{\mathrm{dR},\log}(f)=2.
\]
\end{proposition}

\begin{proof}
The fixed-point algebra and the admissible presentation are given in \Cref{ex:hidden-initial-relation}. In the logarithmic frame \((\dd z/z,\dd w)\),
\[
f^*\left(\frac{\dd z}{z}\right)
=
\frac{\dd z}{z}-\frac{\dd w}{1-w},
\qquad
f^*(\dd w)=2z^2\frac{\dd z}{z}.
\]
Hence
\begin{equation*}
\det(1+yJ_f^{\log})
=
1+y+\frac{2y^2z^2}{1-w}.
\end{equation*}
For the presentation \(q_1=z^3\), \(q_2=w-z^2\), \(e(q)=2\), \(\Delta_q=1\), and
\[
Q_1=z^3,
\qquad
Q_2=w-t^2z^2.
\]
The relative residue is
\begin{align*}
\mathscr R_{q,\,\det(1+yJ_f^{\log})(tz,w)}(t)
&=
\Res_0
\left[
\frac{1+y+2y^2t^2z^2/(1-w)}{z^3,w-t^2z^2}
\,\dd z\wedge\dd w
\right]\\
&=2y^2t^2.
\end{align*}
The constant term \(1+y\) has zero residue with denominator \(z^3\), whereas the remaining term has residue \(2y^2t^2\). Formula \eqref{eq:hidden-local-polynomial} follows from \Cref{thm:normal-cone-specialisation}.
\end{proof}

The next result gives a coupled family in every dimension and an exact contact formula.

\begin{proposition}\label{prop:higher-dimensional-contact}
Let \(m\geq1\), let \(r\geq1\), and let \(m_1,\ldots,m_m\geq2\). On
\[
(\C^{m+1},0),
\qquad
D=\{z=0\},
\]
with coordinates \((z,y_1,\ldots,y_m)\), define a germ \(f\) by
\begin{align*}
 f_z
 &=z-z^{r+1}(1+y_1),\\
 f_{y_1}
 &=y_1-y_1^{m_1}-z,\\
 f_{y_j}
 &=y_j-y_j^{m_j}-y_{j-1},
 \qquad 2\leq j\leq m.
\end{align*}
Put \(M=m_1\cdots m_m\). Then the origin is an isolated normally resonant fixed point, and
\begin{equation}\label{eq:family-local-algebra}
A_{f,0}
\simeq
\C\{y_m\}/\bigl(y_m^{(r+1)M}\bigr).
\end{equation}
In particular,
\[
\length(A_{f,0})=(r+1)M.
\]
The presentation obtained by removing the unit \(1+y_1\) from the first fixed equation is normally admissible, and the logarithmic de Rham index is
\begin{equation}\label{eq:higher-dimensional-contact-index}
\Ind^{\mathrm{dR},\log}_{0,D}(f)=rM.
\end{equation}
\end{proposition}

\begin{proof}
Write
\[
a(z,y_1)=1-z^r(1+y_1).
\]
Then \(f_z=za\) and \(a(0)=1\), so \(a\) is a unit and \(f\) is strict along \(D\). Moreover, \(\nu_D(1-a)=r\), so the normal contact order is \(r\). The induced normal multiplier at the origin is one, hence the fixed point is normally resonant.
The fixed-point equations are
\begin{align*}
 g_0&=z^{r+1}(1+y_1),\\
 g_1&=y_1^{m_1}+z,\\
 g_j&=y_j^{m_j}+y_{j-1},
 \qquad 2\leq j\leq m.
\end{align*}
Since \(1+y_1\) is a unit, the fixed-point ideal is generated by
\begin{equation*}
\bigl(
 z^{r+1},
 y_1^{m_1}+z,
 y_2^{m_2}+y_1,
 \ldots,
 y_m^{m_m}+y_{m-1}
\bigr).
\end{equation*}
The last \(m\) equations eliminate \(z,y_1,\ldots,y_{m-1}\) successively. Up to a non-zero scalar, the image of \(z\) is \(y_m^M\). The remaining equation is therefore \(y_m^{(r+1)M}=0\), which proves \eqref{eq:family-local-algebra} and the asserted length.
Use the presentation
\[
q_0=z^{r+1},
\qquad
q_j=g_j\quad(1\leq j\leq m).
\]
Then \(q=A g\) with \(\Delta_q=\det A=(1+y_1)^{-1}\). Its normal orders are
\[
d_0=r+1,
\qquad
d_1=\cdots=d_m=0,
\]
so the excess order in \eqref{eq:excess-order} is \(e(q)=r\). The rescaled presentation is
\begin{align*}
Q_0&=z^{r+1},\\
Q_1&=y_1^{m_1}+tz,\\
Q_j&=y_j^{m_j}+y_{j-1},
\qquad 2\leq j\leq m.
\end{align*}
Its special fibre is defined by
\begin{equation}\label{eq:family-special-fibre}
\bigl(
 z^{r+1},
 y_1^{m_1},
 y_2^{m_2}+y_1,
 \ldots,
 y_m^{m_m}+y_{m-1}
\bigr).
\end{equation}
This ideal has finite colength \((r+1)M\), so the presentation is normally admissible.

To compute the numerator at \(y=-1\), use the logarithmic cotangent frame
\[
\left(
\frac{\dd z}{z},\dd y_1,\ldots,\dd y_m
\right).
\]
In this frame, the matrix of \(1-J_f^{\log}\), up to transposition, is
\begin{equation*}
\begin{pmatrix}
\dfrac{rz^r(1+y_1)}{a}
&\dfrac{z^r}{a}
&0&\cdots&0
\\[2mm]
z
&m_1y_1^{m_1-1}
&0&\cdots&0
\\
0&1&m_2y_2^{m_2-1}&&0
\\
\vdots&&\ddots&\ddots&
\\
0&\cdots&0&1&m_my_m^{m_m-1}
\end{pmatrix}.
\end{equation*}
Consequently,
\begin{equation*}
\det(1-J_f^{\log})
=
\frac{z^r}{a}
\left(
rm_1(1+y_1)y_1^{m_1-1}-z
\right)
\prod_{j=2}^{m}m_jy_j^{m_j-1}.
\end{equation*}
After replacing \(z\) by \(tz\) and multiplying by \(\Delta_q=(1+y_1)^{-1}\), the transformed numerator has the form
\[
\det(1-J_f^{\log})(tz,y)\Delta_q(tz,y)
=
t^r\widetilde h(t,z,y),
\]
where
\begin{equation*}
\widetilde h(0,z,y)
=
r z^r
\prod_{j=1}^{m}m_jy_j^{m_j-1}.
\end{equation*}
Since \(e(q)=r\), \Cref{cor:normal-cone-class} gives
\begin{align}\label{eq:family-special-residue}
\Ind^{\mathrm{dR},\log}_{0,D}(f)
&=
\Res_0
\left[
\frac{
 r z^r\prod_{j=1}^{m}m_jy_j^{m_j-1}
 \,\dd z\wedge\dd y_1\wedge\cdots\wedge\dd y_m
}{
 z^{r+1},y_1^{m_1},y_2^{m_2}+y_1,\ldots,
 y_m^{m_m}+y_{m-1}
}
\right].
\end{align}
The Jacobian determinant of the regular sequence in the denominator of \eqref{eq:family-special-residue} is
\[
(r+1)z^r\prod_{j=1}^{m}m_jy_j^{m_j-1}.
\]
By the standard Jacobian formula for Grothendieck residues, the residue of the Jacobian of a zero-dimensional complete intersection is its length; see \cite[Chapter~5]{GriffithsHarris}. Since the quotient defined by \eqref{eq:family-special-fibre} has length \((r+1)M\), equation \eqref{eq:family-special-residue} gives
\[
\Ind^{\mathrm{dR},\log}_{0,D}(f)
=
\frac{r}{r+1}(r+1)M
=rM.
\]
\end{proof}

The factor \(r\) is the order of the normal multiplier's contact with the identity, while
\[
M=
\length
\frac{\C\{y_1,\ldots,y_m\}}
{(y_1^{m_1},y_2^{m_2}+y_1,\ldots,y_m^{m_m}+y_{m-1})}
\]
is the multiplicity of the tangential special fibre. Thus \eqref{eq:higher-dimensional-contact-index} is precisely the product of the normal contact order and the tangential multiplicity, although the original fixed equations are not separated.

\begin{example}\label{ex:coupled-resonance}
For \(m=1\), \(r=1\), and \(m_1=2\), \Cref{prop:higher-dimensional-contact} gives
$
f(z,w)
=
\bigl(z(1-z(1+w)),\,w-w^2-z\bigr).
$
The local fixed-point algebra is
\[
A_{f,0}\simeq\C\{w\}/(w^4),
\]
and the coordinate special fibre is defined by \((z^2(1+w),w^2)\), equivalently by \((z^2,w^2)\). With \(a=1-z(1+w)\), the exact logarithmic determinant is
\[
\det(1-J_f^{\log})
=
\frac{z\bigl(2w(1+w)-z\bigr)}{a}.
\]
The proposition yields
\[
\Ind^{\mathrm{dR},\log}_{0,D}(f)=2.
\]
Because the tangential fixed equation \(w^2+z\) contains the normal variable, this calculation does not follow from \Cref{prop:split-contact}.
\end{example}

\section{The logarithmic de Rham specialisation}\label{sec:de-rham}

Let \(U=X\setminus D\), and suppose first that \(E=\cO_X\). The exterior differential makes
\begin{equation*}
\Omega_X^\bullet(\log D)
=
\left[
\cO_X\xrightarrow{\dd}
\Omega_X^1(\log D)\xrightarrow{\dd}\cdots\xrightarrow{\dd}
\Omega_X^n(\log D)
\right]
\end{equation*}
a complex. Since pull-back by \(f\) commutes with \(\dd\), it acts on its hypercohomology.

The logarithmic Poincar\'e lemma gives a quasi-isomorphism
\begin{equation}\label{eq:log-comparison}
\Omega_X^\bullet(\log D)\simeq \mathbf R j_*\C_U.
\end{equation}
Here \(j:U\hookrightarrow X\); see \cite{DeligneHodgeII,DeligneRegular}. Thus
\[
\mathbb H^q(X,\Omega_X^\bullet(\log D))\simeq H^q(U,\C).
\]

\begin{theorem}\label{thm:de-rham-lefschetz}
Under the hypotheses of \Cref{thm:y-lefschetz}, with \(E=\cO_X\), put \(g_i=x_i-f_i\) in local coordinates at each fixed point. Then
\begin{equation}
\begin{aligned}\label{eq:de-rham-lefschetz}
\sum_j(-1)^j\Tr\bigl(f^*|H^j(U,\C)\bigr)
&=
L_{-1}^{\log}(f;D)
\\
&=
\sum_{p\in\Fix(f)}
\Res_p
\left[
\frac{
\det(1-J_f^{\log})\,
\dd x_1\wedge\cdots\wedge\dd x_n
}{g_1,\ldots,g_n}
\right].
\end{aligned}
\end{equation}
\end{theorem}

\begin{proof}
The hypercohomology spectral sequence
\[
E_1^{r,q}=H^q\bigl(X,\Omega_X^r(\log D)\bigr)
\Longrightarrow
\mathbb H^{r+q}\bigl(X,\Omega_X^\bullet(\log D)\bigr)
\]
is \(f^*\)-equivariant. The alternating trace is unchanged from one page of a finite spectral sequence to the next. The comparison \eqref{eq:log-comparison} therefore identifies its value on the first page with the Lefschetz number of \(f|_U\). The second equality is \Cref{thm:y-lefschetz} at \(y=-1\).
\end{proof}

Denote the first expression in \eqref{eq:de-rham-lefschetz} by \(\Lef(f|_U)\).

Combining \Cref{cor:boundary-vanishing} with \Cref{thm:de-rham-lefschetz} gives a fixed-point formula for the complement in which normally non-resonant points of the compactifying boundary disappear. Set
\[
\Fix_{\mathrm{res}}(f;D)
=
\{p\in\Fix(f)\cap D:1\in\Spec(\dd_Nf_p)\},
\]
and, for \(p\in U\), define
\[
\Ind_p^{\mathrm{dR}}(f)
=
\Res_p
\left[
\frac{\det(1-J_f)\,\dd x_1\wedge\cdots\wedge\dd x_n}
{x_1-f_1,\ldots,x_n-f_n}
\right].
\]

\begin{corollary}\label{cor:open-boundary-formula}
The Lefschetz number of the complement is
\begin{align*}
\Lef(f|_U)
&=
\sum_{p\in\Fix(f)\cap U}
\Ind_p^{\mathrm{dR}}(f)
+
\sum_{p\in\Fix_{\mathrm{res}}(f;D)}
\Ind^{\mathrm{dR},\log}_{p,D}(f).
\end{align*}
If an interior fixed point is non-degenerate, its contribution is one.
\end{corollary}

\begin{proof}
Normally non-resonant boundary terms vanish by \Cref{cor:boundary-vanishing}. At an interior point, \(J_f^{\log}=J_f\), so the local term is \(\Ind_p^{\mathrm{dR}}(f)\). In the non-degenerate case it equals
\[
\frac{\det(1-\dd f_p^*)}{\det(1-\dd f_p)}=1.
\]
The remaining boundary terms are precisely those indexed by \(\Fix_{\mathrm{res}}(f;D)\).
\end{proof}

More generally, let \((E^\circ,\nabla^\circ)\) be a regular singular flat bundle on \(U\), let \(\mathbb V\) be its local system of horizontal sections, and let \((E,\nabla)\) be the Deligne canonical extension with integrable logarithmic connection
\[
\nabla:E\longrightarrow E\otimes\Omega_X^1(\log D).
\]
If, in addition, a horizontal morphism \(\Phi:f^*(E,\nabla)\to(E,\nabla)\) of the logarithmic extensions is given, the same spectral-sequence trace argument applies to the logarithmic de Rham complex \(\DR_{\log}(E,\nabla)\), whose hypercohomology computes the cohomology of \(U\) with coefficients in \(\mathbb V\); see \cite{DeligneRegular}. Coefficient systems are not considered further here.

\section{Global examples}\label{sec:examples}

Three compact examples exhibit resonant boundary contributions and non-trivial fixed-point formulae for the complement.

\begin{example}\label{ex:projective-line}
Let
$
X=\mathbb P^1, \ 
D=\{[1:0]\},
$
and choose \(c\in\C^\times\). In the affine coordinate \(z=Z_1/Z_0\), consider the automorphism
\begin{equation*}
f_c(z)=\frac{z}{1+cz},
\qquad
f_c([Z_0:Z_1])=[Z_0+cZ_1:Z_1].
\end{equation*}
The divisor \(D=\{z=0\}\) is strictly preserved. The fixed equation is
\[
z-f_c(z)=\frac{cz^2}{1+cz},
\]
so the only fixed point is the resonant boundary point \(z=0\), with contact order one. The logarithmic cotangent multiplier is
\[
J_{f_c}^{\log}(z)=\frac{zf_c'(z)}{f_c(z)}=\frac{1}{1+cz}.
\]
Consequently,
\begin{align*}
\Ind_{0,D}^{\mathrm{bdry}}(f_c;y)
&=
\Res_0
\left[
\frac{1+y/(1+cz)}{cz^2/(1+cz)}\,\dd z
\right]
=1.
\end{align*}
Globally, \(\Omega_{\mathbb P^1}^1(\log D)\simeq\cO_{\mathbb P^1}(-1)\), which has no cohomology, and hence
\[
L_y^{\log}(f_c;D)=1.
\]
On \(U=X\setminus D\), the coordinate \(u=1/z\) identifies \(f_c\) with the translation \(u\mapsto u+c\). Thus \(\Lef(f_c|_U)=1\), in agreement with \Cref{prop:curve-contact}.
\end{example}

\begin{example}\label{ex:product-surface}
Let
$
X=\mathbb P^1_z\times\mathbb P^1_w,
\
D=\{z=0\}\times\mathbb P^1_w,
$
and let
\[
f(z,w)=\left(\frac{z}{1+cz},\,\beta w\right),
\qquad
c\in\C^\times,\quad \beta\in\C^\times\setminus\{1\}.
\]
There are exactly two fixed points, \((0,0)\) and \((0,\infty)\), both resonant boundary points. The map splits into the parabolic normal factor of \Cref{ex:projective-line} and a non-degenerate tangential factor. Their local polynomials are
\[
\frac{1+y\beta}{1-\beta},
\qquad
\frac{1+y\beta^{-1}}{1-\beta^{-1}},
\]
and their sum is
\begin{equation}\label{eq:surface-local-sum}
1-y.
\end{equation}
On the other hand, the K\"unneth formula makes the logarithmic \(y\)-Lefschetz polynomial multiplicative for this product. The first factor contributes one by \Cref{ex:projective-line}, while the second contributes \(1-y\). Therefore
\[
L_y^{\log}(f;D)=1-y,
\]
which agrees with \eqref{eq:surface-local-sum}. At \(y=-1\), the complement is \(\C\times\mathbb P^1\), and the Lefschetz number is two. Each resonant boundary point contributes one.
\end{example}

\begin{example}\label{ex:mixed-blow-up}
Let \(A:\mathbb P^2\to\mathbb P^2\) be the automorphism
$A[X:Y:Z]=[X+Y:Y:\lambda Z],
\ \lambda\in\C^\times\setminus\{1\},
$
and let \(p=[1:0:0]\). Put
$
X=\operatorname{Bl}_p\mathbb P^2 $ and $
D=\pi^{-1}(p).
$
The automorphism lifts to a strict self-map \(f:(X,D)\to(X,D)\). Its fixed-point set consists of the interior point
$
q=[0:0:1]
$
and two points \(p_0,p_\infty\) on the exceptional divisor, corresponding, respectively, to the eigendirections associated with the eigenvalues \(1\) and \(\lambda\) in \(T_p\mathbb P^2\).
Near \(p\), use affine coordinates \(x=Y/X\) and \(z=Z/X\). Then
\[
A(x,z)=\left(\frac{x}{1+x},\frac{\lambda z}{1+x}\right).
\]
In the blow-up chart \(x=u\), \(z=uv\), the lifted map is
\[
(u,v)\longmapsto\left(\frac{u}{1+u},\lambda v\right).
\]
Thus \(p_0=(0,0)\) is resonant in the normal direction, and the product calculation gives
\begin{equation}\label{eq:mixed-p0}
\Ind_{p_0,D}^{\mathrm{bdry}}(f;y)
=
\frac{1+y\lambda}{1-\lambda}.
\end{equation}
In the second chart \(z=s\), \(x=st\), the lifted map is
\[
(s,t)\longmapsto
\left(\frac{\lambda s}{1+st},\lambda^{-1}t\right).
\]
The point \(p_\infty=(0,0)\) is normally non-resonant, and \Cref{cor:nondegenerate-formula} gives
\begin{equation*}
\Ind_{p_\infty,D}^{\mathrm{bdry}}(f;y)
=
\frac{(1+y)(1+y\lambda^{-1})}
{(1-\lambda)(1-\lambda^{-1})}.
\end{equation*}
Finally, in affine coordinates \(a=X/Z\), \(b=Y/Z\) centred at \(q\), the differential is the Jordan matrix
\[
\lambda^{-1}
\begin{pmatrix}
1&1\\
0&1
\end{pmatrix}.
\]
Hence \(q\) is non-degenerate and
\begin{equation}\label{eq:mixed-q}
\Ind_q^{\mathrm{int}}(f;y)
=
\frac{(1+y\lambda^{-1})^2}{(1-\lambda^{-1})^2}.
\end{equation}
A direct simplification of \eqref{eq:mixed-p0}--\eqref{eq:mixed-q} yields
\begin{equation}\label{eq:mixed-local-sum}
\Ind_q^{\mathrm{int}}(f;y)
+
\Ind_{p_0,D}^{\mathrm{bdry}}(f;y)
+
\Ind_{p_\infty,D}^{\mathrm{bdry}}(f;y)
=1-y.
\end{equation}

The complement \(U=X\setminus D\) is \(\mathbb P^2\setminus\{p\}\), and projection from \(p\) identifies it with the total space of an algebraic affine-line bundle over \(\mathbb P^1\). Hence pull-back identifies its cohomology, together with its mixed Hodge structure, with that of \(\mathbb P^1\); in particular, \(H^2(U,\C)\) is one-dimensional of type \((1,1)\). The induced automorphism acts trivially on \(H^0(U,\C)\) and \(H^2(U,\C)\). Since \(X\) is projective, Deligne's logarithmic Hodge-to-de Rham spectral sequence degenerates at its first page \cite{DeligneHodgeII}; the only non-zero graded pieces therefore occur in bidegrees \((0,0)\) and \((1,1)\), each with trace one. Therefore
\[
L_y^{\log}(f;D)=1-y,
\]
in agreement with \eqref{eq:mixed-local-sum}. At \(y=-1\), the interior point and the resonant boundary point contribute one each, whereas the normally non-resonant boundary point contributes zero. Thus \(\Lef(f|_U)=2\).
\end{example}

\section*{Acknowledgements}
The author thanks Mauricio Corrêa for numerous insightful and useful conversations.

\end{document}